\documentclass{amsart}

\usepackage{comment}

\usepackage[english]{babel}
\usepackage{amsmath,amssymb,amsbsy,amsthm,amsfonts,enumerate,array,color,lscape,fancyhdr,layout,pst-all}
\usepackage[pdftex]{hyperref}
\usepackage[all]{xy}
\usepackage[utf8]{inputenc}
\usepackage{stmaryrd}
\usepackage{blkarray}
\usepackage{tikz}
 \usetikzlibrary{matrix,decorations.pathreplacing, calc, positioning}
 \pgfkeys{tikz/mymatrixenv/.style={decoration=brace,every left delimiter/.style=    {xshift=3pt},every right delimiter/.style={xshift=-3pt}}}
 \pgfkeys{tikz/mymatrix/.style={matrix of math nodes,left delimiter=[,right delimiter=    {]},inner sep=2pt,column sep=1em,row sep=0.5em,nodes={inner sep=0pt}}}
 \pgfkeys{tikz/mymatrixbrace/.style={decorate,thick}}

\newtheorem{theorem}{Theorem}[section]
\newtheorem{lemma}[theorem]{Lemma}
\newtheorem{corollary}[theorem]{Corollary}
\newtheorem{proposition}[theorem]{Proposition}
\usepackage{amssymb}
\usepackage{setspace}
\newtheorem{definition}[theorem]{Definition}
\newtheorem{example}[theorem]{Example}
\newtheorem{remark}[theorem]{Remark}
\newtheorem{notation}[theorem]{Notation}

\title[Triangulated structure for Bondarenko's Categories]{Triangulated structure for Bondarenko's Categories}
\date{\today}

\author[Germán Benitez]{Germán Benitez}
\address{\noindent Departamento de Matem\'atica, Instituto de Ciências Exatas, Universidade Federal do Amazonas,  Manaus AM, Brazil}
\email{gabm03@gmail.com}

\author[Gustavo Costa]{Gustavo Costa}
\address{\noindent Departamento de Matem\'atica, Instituto de Ciências Exatas, Universidade Federal do Amazonas,  Manaus AM, Brazil}
\email{costagustt@gmail.com}

\author[Lucas Queiroz Pinto]{Lucas Q. Pinto}
\address{\noindent Instituto de Matem\'atica e Estadistica, Universidade de São Paulo,  Sao Paulo SP, Brazil}
\email{lucasqueirozp94@gmail.com}

\begin{document}

\begin{abstract}

V. Bondarenko and Y. Drozd gives a description of all indecomposable objects in a category of representations of posets, nowadays known as the Bondarenko’s category. This category was essential for V. Bekkert and H. Merklen classify all indecomposable objects of the derived category of gentle algebras. In view of this connection with the derived category, which possess a triangulated structure, and of the fact that in this paper we show that the Bondarenko’s category is not an abelian category, it is reasonable to contemplate the existence of a triangulated structure for the Bondarenko's category.  

In this paper we introduce a triangulated category structure over a quotient of Bondarenko’s category, which will allow to use the techniques of triangulated category to study representations of posets.

\end{abstract}
\subjclass{Primary 16G20, 18G80}

\keywords{Representations of poset, triangulated category}

\date{}


\maketitle

\tableofcontents    


\maketitle


\section{Introduction}

I. M. Gelfand in the International Congress of Mathematicians which was held at Nice in 1970 proposed to obtain a description of the indecomposable representations of the quiver (see \cite{Gel71})
	$$
	\xymatrix{ \bullet\ar@/^/[r]^{b} & \bullet\ar@/^/[l]^{a}\ar@/_/[r]_{c} & \bullet\ar@/_/[l]_{d}}\ \ \ \ \ \ \ \text{with $ba=dc$,}
	$$
in connection with the classification of Harish-Chandra-modules at a given singular point for $\textup{SL}_2(\mathbb{R})$. This problem was solved in 1973 by L. A. Nazarova and A. V. Roiter (see \cite{NR73}) by reducing it to a certain matrix problem. The matrix problems studied by V. Bondarenko and Y. Drozd around to 80's (see \cite{Bon75,BD82}) allowed to V. Bondarenko in 1988 to reconsider the problem of Gelfand (cf. \cite{Bon88,Bon91}) and introduces a wider class of matrix problems than in \cite{NR73} consisting on block matrices indexed by poset with involution (called $\mathcal{S}$-representation or representation of poset with involution). From these works, in 2003, V. Bekkert and H. Merklen presented the Bondarenko's category (also, called category of $\mathcal{S}$-representations) to classify the indecomposable objects of the derived category of gentle algebras (see \cite{BM03}). The category of $\mathcal{S}$-representations was introduced in \cite{BD82} and is characterized by objects and morphisms represented by block matrices indexed by a poset with involution, and the objects themselves were initially introduced in \cite{Bon75}.

In \cite{BM03}, V. Bekkert and H. Merklen have constructed a functor arriving into certain Bondarenko's category with a particular poset and involution. This approach has been adapted to classify the indecomposable objects of the derived category of Skewed-Gentle algebras in \cite{BNM03}, also for certain Algebra of Dihedral Type in \cite{GV16} and took advantage of this to discuss the shape of the corresponding components of the Auslander-Reiten quiver containing these objects. Moreover, the techniques used by V. Bekkert and H. Merklen in \cite{BM03} were the starting point to research new classes of algebras in which is possible to describe indecomposable objects of its derived category, such algebras was introduced in  \cite{FGR21} and called string almost gentle (SAG) algebras and SUMP algebras.

It is well-known that the category of representations of poset is an additive category, in Remark~\ref{rem:nonabel} we explain why the category is not an abelian category in general, specifically, the category does not always have kernel or cokernel. It is natural to think that a category without kernel or cokernel has few good tools. Therefore, given the importance of Bondarenko's category, it makes sense to question if there exists some triangulated structure that allows us to study representation of posets. In this paper we give a triangulated category structure over a quotient of Bondareko's category for an arbitrary poset with involution, which will allow to use the techniques of triangulated category to study representations of posets with involution, for instance, five lemma and the cohomology functor $\textup{Hom}(-,-)$ for triangulated category.

This paper is organized as follows. In Section~\ref{sec:Bond} we recall the definition of Bondarenko's category, also we discuss some facts about the non-abelian structure and we introduce the standard triangles which will be useful to define the family of distinguished triangles. Section~\ref{sec:triang} is dedicated to prove that certain quotient of Bondarenko's category is triangulated starting to introduce the quotient category (which studies morphism up to certain equivalence relation) and the family of distinguished triangles.

\section{Bondarenko's Category}
\label{sec:Bond}
In this paper, $\Bbbk$ will denote a field and $\mathcal{Y}$ a linearly ordered set (may be infinite) endowed with an involution $\sigma$. For a block matrix indexed by $\mathcal{Y}$, say $B=(B^{j}_i)_{i,j\in \mathcal{Y}}$ (not necessarily square blocks) with all the entries of all the blocks sitting in $\Bbbk$, the subindex denotes the block row and the superindex denotes the block column. Additionally, we write $B_x$ for the horizontal band and $B^x$ the vertical band (Notice also that some blocks may be empty).


Given two matrices $C=(C^j_i)_{i,j\in \mathcal{Y}}$ and $D=(D^j_i)_{i,j \in \mathcal{Y}}$ (not necessarily square), we say that each  horizontal partition $C_x$ of $C$ is \emph{compatible} with each vertical partition $D^x$ of $D$ if the number of  rows in each $C_x$ is equal to the number of columns in each $D^x$. Similarly, we can define the compatibility between the vertical partitions of $C$ and the horizontal partitions of $D$. Note that these notions of compatibility allow us to do the multiplication $DC$ and $CD$ by blocks, respectively.

Let us define the category $s(\mathcal{Y}, \Bbbk)$ known as \emph{Bondarenko's category of representations of posets} (see \cite{BNM03,BM03,FGR21}). 

\begin{definition}

The objects of $s(\mathcal{Y}, \Bbbk)$ are finite square block matrices $B=(B^{j}_i)_{i,j\in \mathcal{Y}}$  called \emph{representations} or \emph{$\mathcal{Y}$-matrices} satisfying the following conditions. 
\begin{enumerate}[(i)]
\item The horizontal and vertical partitions by blocks of $B$ are compatible. 

\item If $i,j \in \mathcal{Y}$ are such that $\sigma(i)=j$, then all matrices in $B_i$ (resp. $B^i$) have the same number of rows (resp. columns) as all matrices in $B_j$
 (resp. $B^j$).
\item $B^2=0$.
\end{enumerate}

A morphism in 
$s(\mathcal{Y},\Bbbk)$ from $B$ to $C$ is a block matrix $T=(T^{j}_{i})_{i,j \in \mathcal{Y}}$, with entries in $\Bbbk$ such that the following conditions hold.
\begin{enumerate}[(a)]
 \item The horizontal (resp. vertical) partition of $T$ is compatible with the vertical (resp. horizontal) partition of $B$ (resp. $C$). 
\item $TC=BT$.
\item If $i>j$, then $T^j_i=0$, where $<$ is the order relation in the poset $\mathcal{Y}$, i.e., all blocks below the main diagonal are 0.
\item If $\sigma(i)=j$, then $T^i_i=T^j_j$.
\end{enumerate}

\end{definition}

Since the matrices $T$ are upper triangular, in order to have an isomorphism $T$ in the category $s(\mathcal{Y},\Bbbk)$, it is necessary and sufficient that the diagonal blocks $T^i_i$  of $T$ be all invertible. 

Since some blocks may be empty, the zero object in $s(\mathcal{Y}, \Bbbk)$ corresponds to the matrix where all the blocks are empty, which will be denoted by $\mathbb{O}$. And, as  $\textup{Hom}_{s(\mathcal{Y},\Bbbk)}(B,C)$ is a $\Bbbk$-vector space,  $s(\mathcal{Y},\Bbbk)$ is a preadditive $\Bbbk$-category. We will denote the null morphism in $\textup{Hom}_{s(\mathcal{Y},\Bbbk)}(B,C)$ by $\mathbf{0}_B^C$ or simplying $\mathbf{0}$ when no confusion can arise. In \cite{BD82}, the authors stated that $s(\mathcal{Y},\Bbbk)$ is an additive $\Bbbk$-category. In Proposition~\ref{prop:soma}, we will show explicitly the direct sum in $s(\mathcal{Y},\Bbbk)$, to this end, we introduce the following notations, which will be also useful to the next sections.


\vspace{0,5cm}

In what follows, we provide some useful facts about Bondarenko's category. For instance, we exhibit the direct sum of objects, we show that it is not an abelian category and, finally, introduce the standard triangle for Bondarenko's category.

 \begin{notation}
 \label{nota:block} 
 
For any morphism $T\in\textup{Hom}_{s(\mathcal{Y},\Bbbk)}(B,C)$, we define the matrix
    $
\mathbf{C}_T=\left(\left(\mathbf{C}_T\right)_i^j\right)_{i,j \in \mathcal{Y}},
    $
where the each $(i,j)th$ block is the form
    $$
    \left(\mathbf{C}_T\right)_i^j=\left(\begin{array}{cc}
    -B^j_i &T^j_i  \\
    (\mathbf{0}^B_C)_i^j & C^j_i
    \end{array}\right).
    $$
Write $I_D:=\left\{x\in\mathcal{Y}\mid D_x\neq\emptyset\right\}$ for each $D$ square block matrices indexed by $\mathcal{Y}$. Therefore, for each $i,j\in I_{\mathbf{C}_T}=I_B\cup I_C$ we can visualize such blocks in the following table

\vspace{.3cm}

\begin{center}
\begin{tabular}{|c||c|c|c|}
\hline
\mbox{$\left(\begin{array}{cc}
    -B^j_i &T^j_i  \\
    (\mathbf{0}^B_C)_i^j & C^j_i
    \end{array}\right)$}  & \mbox{$j\in I_B\cap I_C$} & \mbox{$j\in I_B\setminus I_C$} &\mbox{$j\in I_C\setminus I_B$} \\
\hline
\hline
$i\in I_B\cap I_C$ & $\left(\begin{array}{cc}
    -B^j_i &T^j_i  \\
    (\mathbf{0}^B_C)_i^j & C^j_i
    \end{array}\right)$	& $\left(\begin{array}{cc}
-B^j_i \\
(\mathbf{0}^B_C)_i^j	
\end{array}\right)$ 	&		$\left(\begin{array}{cc}
T^j_i\\ 
 C^{j}_{i}
\end{array}\right)$	  \\
&&&\\
$i\in I_B\setminus I_C$     				&  $\left(\begin{array}{cc}
-B^j_i & T^j_i\\	
\end{array}\right)$	 & $\left(\begin{array}{cc}
-B^j_i 
\end{array}\right)$	 & 		$\left(T^j_i\right)$	 \\ 
&&&\\
$i\in I_C\setminus I_B$ 			&  $\left(\begin{array}{cc}
(\mathbf{0}^B_C)_i^j	& C^{j}_{i}	
\end{array}\right)$ &$(\mathbf{0}^B_C)_i^j$&$\left(\begin{array}{cc}
C^{j}_{i}	
\end{array}\right)$ \\

\hline
\end{tabular}
\end{center}

\vspace{.3cm}

Note that for the identity morphism $\textup{Id}_B:B\longrightarrow B$, all the blocks of the matrix $\mathbf{C}_{\textup{Id}_B}$ are given by
$$
\left(\begin{array}{cc}
    -B^j_i & (\textup{Id}_B)^j_i  \\
     \mathbf{0}_i^j & B^j_i
\end{array}\right).
$$
\end{notation}


\begin{lemma}\label{lem:CT}
 $\mathbf{C}_T$ is an object in $s(\mathcal{Y},\Bbbk)$, for any morphism $T\in\textup{Hom}_{s(\mathcal{Y},\Bbbk)}(B,C)$. 
\end{lemma}


\begin{proof}
The properties of $\mathbf{C}_T$ involving the compatibility of partitions and involution are consequence to the fact that $B$ and $C$ have those properties. Now, $\mathbf{C}_T^2=0$, follows from $B^2=0$, $C^2=0$, $TC=BT$ and
   {\footnotesize{ \begin{align*}
    (\mathbf{C}_T^2)^j_i & =\left(\sum_{k\in \mathcal{Y}}\left(\begin{array}{cc}
    -B^k_i& T^k_i  \\
     (\mathbf{0}_C^B)^k_i & C^k_{i}
\end{array}\right) \left(\begin{array}{cc}
    -B^j_k &T_k^j  \\
     (\mathbf{0}_C^B)^j_k & C^j_{k}
 \end{array}\right)\right)
 = \left(\begin{array}{cc}
   \sum\limits_{k\in \mathcal{Y}} B^k_i B_k^j  &\sum\limits_{k\in \mathcal{Y}}(T^k_i C^j_k-B^k_iT^j_k)  \\
    (\mathbf{0}^B_C)^j_i &\sum\limits_{k\in \mathcal{Y}} C^k_{i}C^j_k
\end{array} \right)\\
&=  \left(\begin{array}{cc}
   (B^2)_i^j  & (TC-BT)^j_i  \\
    (\mathbf{0}^B_C)^j_i & (C^2)^j_i
\end{array} \right) 
 =  \left(\begin{array}{cc}
 (\mathbf{0}_B^B)^j_i & (\mathbf{0}_B^C)^j_i  \\
 (\mathbf{0}_C^B)^j_i & (\mathbf{0}_C^C)^j_i
\end{array} \right),
    \end{align*}}}
 for all $i,j \in \mathcal{Y}$. Thus, $\mathbf{C}_T$ is an object in $s(\mathcal{Y},\Bbbk)$.

\end{proof}

In terms of Notation~\ref{nota:block}, we will introduce the following matrices associated to each ordered pair of objects $(B,C)$ 
    $$
\widehat{\iota}_B:=\left(\begin{array}{cc}
    (\textup{Id}_{B})^j_i & (\mathbf{0}_B^C)_i^j  \\
\end{array}\right)_{i,j \in \mathcal{Y}}\ \ \ , \ \ \
\iota_C:=\left(\begin{array}{cc}
    (\mathbf{0}_C^B)_i^j & (\textup{Id}_{C})^j_i  \\
\end{array}\right)_{i,j \in \mathcal{Y}},
    $$ 
    $$
    \pi_B:=\left(\begin{array}{cc}
    (\textup{Id}_{B})^j_i\\
    (\mathbf{0}_C^B)_i^j
\end{array}\right)_{i,j \in \mathcal{Y}} \ \ \ \text{ and }\ \ \   \widehat{\pi}_C:=\left(\begin{array}{cc}
(\mathbf{0}_B^C)_i^j\\
    (\textup{Id}_{C})^j_i 
\end{array}\right)_{i,j \in \mathcal{Y}}.$$
For instance, the $(i,j)th$ block
    $$
    (\widehat{\iota}_B)^j_i=\left(\begin{array}{cc}
    (\textup{Id}_{B})^j_i & (\mathbf{0}_B^C)_i^j  \\
\end{array}\right)
    $$
has $r(B_i)$ rows and $c(B^j)+c(C^j)$ columns, where $r(D_x)$ (resp. $c(D^x)$) denote the number of rows (resp. columns) in the horizontal (resp. vertical) band $D_x$ (resp.  $D^x$). Specifically, for $i\in I_B$ and $j\in I_B\cup I_C$

\vspace{.3cm}

\begin{center}
\begin{tabular}{|c||c|c|c|}
\hline
\mbox{$\left(\begin{array}{cc}
    (\textup{Id}_B)^i_i&(\mathbf{0}_B^C)_i^j 
\end{array}\right)$}  & \mbox{$j\in I_B\cap I_C$} & \mbox{$j\in I_B\setminus I_C$} & \mbox{$j\in I_C\setminus I_B$}  \\
\hline
\hline

$i\in I_B$ 			&  $\left(\begin{array}{cc}
 (\textup{Id}_B)^i_i	& (\mathbf{0}_B^C)_i^j
\end{array}\right)$ &$(\textup{Id}_B)^i_i$ & $(\mathbf{0}_B^C)_i^j$\\
\hline
\end{tabular}
\end{center}
Similarly, for the $(i,j)$th blocks of $\iota_C$, $\pi_B$ and $\widehat{\pi}_C$. 

\vspace{.3cm}

The following lemma allows us to check when the latter matrices are morphisms in $s(\mathcal{Y},\Bbbk)$.

\begin{lemma}
\label{lem:morph}
Let $T\in \textup{Hom}_{s(\mathcal{Y},\Bbbk)}(B,C)$ be a morphism. Then,
$$
\iota_C\in \textup{Hom}_{s(\mathcal{Y},\Bbbk)}(C,\mathbf{C}_T)\ \ \ \ \ \text{and }\ \ \ \ \ \pi_B\in \textup{Hom}_{s(\mathcal{Y},\Bbbk)}(\mathbf{C}_T,-B).
$$
Furthermore, $T=\mathbf{0}$ if and only if 
	$$
	\widehat{\iota}_B\in \textup{Hom}_{s(\mathcal{Y},\Bbbk)}(-B,\mathbf{C}_T)\ \ \ \ \ \text{and }\ \ \ \ \widehat{\pi}_C\in \textup{Hom}_{s(\mathcal{Y},\Bbbk)}(\mathbf{C}_T,C).$$
\end{lemma}


\begin{proof}
We only show for $\iota_C$, since for $\pi_B$ the proof is analogous. The properties (a), (c) and (d) of the definition of morphism in $s(\mathcal{Y},\Bbbk)$ follow from the latter construction. The equality $C\iota_C=\iota_C\mathbf{C}_T$ is a consequence of the following computations
{\footnotesize{\begin{align*}
  (  C\iota_C)^j_i &=\sum_{k\in \mathcal{Y}}C^k_i\left(\begin{array}{cc}
    (\mathbf{0}_C^B)^j_k & (\textup{Id}_{C})^j_k \\
\end{array}\right)=C^j_i\left(\begin{array}{cc}
    (\mathbf{0}_C^B)^j_j & (\textup{Id}_{C})^j_j \\
\end{array}\right)=\left(\begin{array}{cc}
    (\mathbf{0}_C^B)_i^j & C^j_i   \\
\end{array} \right)\\
&= \left(\begin{array}{cc}
    (\mathbf{0}_C^B)_i^i & (\textup{Id}_{C})^i_i  \\
\end{array}\right) \left(\begin{array}{cc}
    -B^j_i &T^j_i  \\
    (\mathbf{0}_C^B)_i^j & C^j_{i}
\end{array}\right)=
\sum_{k \in \mathcal{Y}}\left(\begin{array}{cc}
    (\mathbf{0}_C^B)_i^k & (\textup{Id}_{C})^k_i   \\
\end{array}\right)\left(\begin{array}{cc}
    -B^j_k &T^j_k  \\
    (\mathbf{0}_C^B)_k^j & C^j_{k}
\end{array}\right)\\
&=(\iota_C\mathbf{C}_T)^j_i,
\end{align*}}}
 for all $i,j \in \mathcal{Y}$. 

 On the other hand, applying the same line reasoning and the fact that
 
    {\footnotesize{$$
   ( B\widehat{\iota}_B)^j_i=\sum_{k\in \mathcal{Y}}B^k_i\left(\begin{array}{cc}
     (\textup{Id}_{B})^j_k  & (\mathbf{0}_B^C)_k^j \\
\end{array}\right)=B^j_i\left(\begin{array}{cc}
     (\textup{Id}_{B})^j_j & (\mathbf{0}_B^C)_j^j  \\
\end{array}\right)=\left(\begin{array}{cc}
     B^j_i & (\mathbf{0}_B^C)_i^j   \\
\end{array} \right)
    $$}}
    {\footnotesize{\begin{align*}
(\widehat{\iota}_B\mathbf{C}_T)^j_i& =
\sum_{k \in \mathcal{Y}}\left(\begin{array}{cc}
     (\textup{Id}_{B})^k_i & (\mathbf{0}_B^C)_i^k   \\
\end{array}\right)\left(\begin{array}{cc}
    B^j_k &T^j_k  \\
    (\mathbf{0}^B_C)_k^j & C^j_{k}
\end{array}\right)
= \left(\begin{array}{cc}
    (\textup{Id}_{B})^i_i &(\mathbf{0}_B^C)_i^i \\
\end{array}\right)\left(\begin{array}{cc}
    B^j_i &T^j_i  \\
    (\mathbf{0}^B_C)_i^j & C^j_{i}
\end{array}\right) \\
&=\left(\begin{array}{cc}
    B^j_i & T^j_i \\
\end{array} \right),
\end{align*} }}
we have, $\widehat{\iota}_B$ is morphism if and only if $T=\mathbf{0}$. Similarly to $\widehat{\pi}_C$.

\end{proof}


We can now to show that $s(\mathcal{Y},\Bbbk)$ is closed by direct sum. 


\begin{proposition}
\label{prop:soma}
Let B, C be objects in $s(\mathcal{Y}, \Bbbk)$. The direct sum in $s(\mathcal{Y}, \Bbbk)$ between $B$ and $C$, is given by 
    $$
    B\oplus C=\left(\begin{array}{cc}
    B^j_i & (\mathbf{0}_B^C)_i^j  \\
    (\mathbf{0}_C^B)_i^j & C^j_i
    \end{array}\right)_{i,j \in \mathcal{Y}}.
    $$

\end{proposition}


\begin{proof}
Let $B$ and $C$ be objects in $s(\mathcal{Y}, \Bbbk)$. Denote $\mathbf{C}_{\mathbf{0}_{-B}^C}=\left(\begin{array}{cc}
    B^j_i & (\mathbf{0}_B^C)_i^j  \\
    (\mathbf{0}_C^B)_i^j & C^j_i
    \end{array}\right)_{i,j \in \mathcal{Y}}$ by $\mathbf{C}_{\mathbf{0}}$, which is an object in $s(\mathcal{Y}, \Bbbk)$ by Lemma~\ref{lem:CT}. Under the notations of Lemma~\ref{lem:morph}, $\mathbf{C}_{\mathbf{0}}$ with inclusion $\widehat{\iota}_B$, $\iota_C$, and projections $\pi_B$ and $\widehat{\pi}_C$ is the direct sum of $B$ and $C$ in $s(\mathcal{Y}, \Bbbk)$.

By \cite[Lemma 5.87, p. 304]{Rot09} it is sufficient to show 
	$$
	\widehat{\iota}_B\widehat{\pi}_C=\mathbf{0}_B^C,\ \ \  \iota_C\pi_B=\mathbf{0}_C^B,\ \ \  \widehat{\iota}_B\pi_B=\textup{Id}_B,\ \ \  \iota_C\widehat{\pi}_C=\textup{Id}_C\ \ \ \text{and }\ \  \pi_B\widehat{\iota}_B+\widehat{\pi}_C\iota_C=\textup{Id}_{\mathbf{C}_{\mathbf{0}}},$$
which is straightforward by multiplication of matrices.

\end{proof}

\begin{corollary}
$s(\mathcal{Y},\Bbbk)$ is an additive $\Bbbk$-category.
\end{corollary}


For a better exposition, in Remark~ \ref{rem:nonabel} and Example~\ref{ex:triang}, we will use color blue to identify the indices of a matrix block.


\begin{remark}
\label{rem:nonabel}
It is important to note that $s(\mathcal{Y},\Bbbk)$ is not an abelian category. For instance, consider a poset $\mathcal{Y}$ with at most three elements, $\alpha_1<\alpha_2<\alpha_3$, and an involution $\sigma$ such that $\sigma(\alpha_3)=\alpha_3$ and
    $$
    \sigma(\alpha_i)= \alpha_i\ (i=1,2)\ \ \text{or }\ \sigma(\alpha_1)=\alpha_2.
    $$
To obtain a contradiction, suppose that $s(\mathcal{Y},\Bbbk)$ is an abelian category. Is it easy to check that, the kernel and cokernel of the following morphism 
$$
T=\left(\begin{array}{c||c}
     & {\blue \alpha_3}\\
     \hline\hline
    {\blue \alpha_1} & 0\\\hline
    {\blue \alpha_2} & 1
\end{array}\right):\left(\begin{array}{c||c|c}
     &{\blue \alpha_1} & {\blue \alpha_2}\\
     \hline\hline
    {\blue \alpha_1} &0 &0\\\hline
    {\blue \alpha_2} &0&0
\end{array}\right)\longrightarrow \left(\begin{array}{c||c}
     & {\blue \alpha_3}\\
     \hline\hline
    {\blue \alpha_3} & 0
\end{array}\right)$$
are  $\ker(T)=\left(\begin{array}{c||c}
     & {\blue \alpha_1}\\
     \hline\hline
    {\blue \alpha_1} & 0
\end{array}\right)$ with $\left(\begin{array}{c||c|c}
     & {\blue \alpha_1} & {\blue \alpha_2}\\
     \hline\hline
    {\blue \alpha_1} & 1 & 0
\end{array}\right):\ker(T)\longrightarrow\left(\begin{array}{c||c|c}
     &{\blue \alpha_1} & {\blue \alpha_2}\\
     \hline\hline
    {\blue \alpha_1} &0 &0\\\hline
    {\blue \alpha_2} &0&0
\end{array}\right)$ its inclusion and $\textup{coker}(T)=\mathbb{O}$. Since such inclusion is split mono with inverse $\left(\begin{array}{c||c}
     &{\blue \alpha_1} \\
     \hline\hline
    {\blue \alpha_1} & 1\\\hline
    {\blue \alpha_2} & 0
\end{array}\right)$, so $T$ is split epi. This contradicts the fact that there is not a non-null morphism in $s(\mathcal{Y},\Bbbk)$ of the form $
\left(\begin{array}{c||c|c}
     & {\blue \alpha_1} & {\blue \alpha_2} \\
     \hline\hline
    {\blue \alpha_3} & 
\end{array}\right)
$.

\end{remark}


Since our purpose is to show that a certain quotient of the category $s(\mathcal{Y},\Bbbk)$ is triangulated, it is useful to define the notion of \emph{standard triangle} in this context, which we define to be sequence of the form 
	$$
    \xymatrix{
    B\ar[r]^-{T} &
    C\ar[r]^-{\iota_B} &
    \mathbf{C}_{T}\ar[r]^-{\pi_{B}} &
    [B]},
    $$
where $\iota_C$ and $\pi_B$ are the inclusion and projection defined in Lemma~\ref{lem:morph}, and defined autofunctor  
\begin{equation*}
\label{[-]}
     [-]:s(\mathcal{Y},\Bbbk)\longrightarrow s(\mathcal{Y},\Bbbk)
\end{equation*}
given by $[B]=-B$ for objects and $[T]=T$ for morphisms. We will finish this section with one example of standard triangle. 


\begin{example}
\label{ex:triang}
Consider $\mathcal{Y}=
    (\mathbb{ Q},<)$ with involution $\sigma(x)=\frac{1}{x}$, 
    if $x\not=0$ and $\sigma(x)=0$, if $x=0$. Clearly, there exists objects in $s(\mathcal{Y},\Bbbk)$ of the form
    $$
   B= {\small{\left(\begin{array}{c||c|c}
     &{\blue 0} & {\blue 1}\\
     \hline\hline
    {\blue 0} & B_0^0 & B_0^1 \\\hline
    {\blue 1} & B_1^0 & B_1^1
\end{array}\right)}} \ \ \ \ \ \text{and } \ \ \ C={\small{\left(\begin{array}{c||c|c|c}
     & {\blue 0} & {\blue 1/2} & {\blue 2}\\
     \hline\hline
    {\blue 0} & C_0^0 & C_0^{1/2} & C_0^{2}\\\hline
    {\blue 1/2} & C_{1/2}^0 & C_{1/2}^{1/2} & C_{1/2}^2 \\\hline
    {\blue 2} & C_{2}^{0} & C_{2}^{1/2} & C_{2}^{2}
\end{array}\right)}}
     $$ 
     
Note that, any morphism $T:B\longrightarrow C$ in $s(\mathcal{Y},\Bbbk)$ and its $\mathbf{C}_T$ have the form 
	$$
	T={\small{\left(\begin{array}{c||c|c|c}
     &{\blue 0}&{\blue 1/2} & {\blue 2}\\
     \hline\hline
    {\blue 0} &T^0_0 &T^{1/2}_0&T^2_0\\\hline
    {\blue 1} &0&0&T^2_1
 \end{array}\right)} }\ 
\ \ \ \text{and} \ \ \
   \mathbf{C}_T={\small{\left(\begin{array}{c||cc|c|c|c}
     & {\blue 0}&&{\blue 1/2}& {\blue 1} & {\blue 2}\\
     \hline\hline
    {\blue 0} &-B^0_0 &T^0_0&T^{1/2}_0&-B^1_0&T^{2}_0\\
            & 0 &C^0_0&C^{{1}/{2}}_0&0&C^2_0\\\hline
      {\blue 1/2} &0&C^0_{1/2}&C^{1/2}_{1/2}&0 &0\\\hline
    {\blue 1} &-B^0_1&0&0&-B^1_1 &T^2_1\\\hline
       {\blue 2} &0&C^0_2&0&0 &C^2_2
\end{array}\right)}}.
    $$
 Therefore, the standard triangle is 
 
	$$
	\xymatrix{B\ar[rrrr]^{\small{\left(\begin{array}{c||c|c|c}
     &{\blue 0}&{\blue 1/2} & {\blue 2}\\
     \hline\hline
    {\blue 0} &T^0_0 &T^{1/2}_0&T^2_0\\\hline
    {\blue 1} &0&0&T^2_1
\end{array}\right)}}&&&&C\ar[rrrr]^{\small{\left(\begin{array}{c||cc|c|c|c}
     & {\blue 0}&&{\blue 1/2}& {\blue 1} & {\blue 2}\\
     \hline\hline
    {\blue 0} &0&1&0&0&0\\
        \hline
      {\blue 1/2} &0&0& 1&0 &0\\\hline
       {\blue 2} &0&0&0&0 &1
\end{array}\right)}}&&&&\mathbf{C}_T\ar[rrr]^{\small{\left(\begin{array}{c||c|c}
     & {\blue 0}& {\blue 1} \\
     \hline\hline
    {\blue 0} &1 &0\\
            & 0 &0\\\hline
      {\blue 1/2} &0&0 \\\hline
    {\blue 1} &0&1\\\hline
       {\blue 2} &0 &0
\end{array}\right)}}&&&-B.}$$ 
\end{example}


\section{Triangulated structure}
\label{sec:triang}


Let $S,T\in \textup{Hom}_{s(\mathcal{Y},\Bbbk)}(B,C)$ be morphisms. We proceed to define the equivalence relation $\equiv$ on the set  $\textup{Hom}_{s(\mathcal{Y},\Bbbk)}(B,C)$  by $S\equiv T$ if and only if there exists a matrix $K$ (we will call $\kappa$-\emph{matrix}) satisfying:
\begin{enumerate}[(i)]
     \item The horizontal (resp. vertical) partition of $K$ is compatible with the vertical (resp. horizontal) partition of $B$ (resp. $C$).
     \item $S-T=BK+KC$.
     \item If $i>j$, then $K^j_i=0$ for all $i,j \in \mathcal{Y}$.
     \item If $\sigma(i)=j$, then $K_i^i=K_j^j$.
\end{enumerate}

An important point to note is the fact that, in general, $K$ is not necessarily a morphism in $s(\mathcal{Y},\Bbbk)$ by item (ii).


\begin{remark}
\label{rem:TR1-2}
Note that, $\textup{Id}_{\mathbf{C}_{\textup{Id}_B}}\equiv \mathbf{0}$ with $\kappa$-matrix $K=\left(\begin{array}{cc}
    (\mathbf{0}_B^B)_i^j & (\mathbf{0}_B^B)_i^j  \\
    (\textup{Id}_B)^j_i & (\mathbf{0}_B^B)_i^j
\end{array}\right)_{i,j \in \mathcal{Y}}$. Morever, for any standard triangle 
    $
    \xymatrix{
    B\ar[r]^-{T} &
    C\ar[r]^-{\iota_C} &
    \mathbf{C}_{T}\ar[r]^-{\pi_{B}} &
    [B]}
    $
we have $\iota_C\pi_B=\mathbf{0}$ and $T\iota_C\neq \mathbf{0}$. However, $T\iota_C\equiv \mathbf{0}$ with $\kappa$-matrix $K=\left(\begin{array}{cc}
    (\textup{Id}_{B})^j_i& (\mathbf{0}_B^C)_i^j  \\
\end{array}\right)_{i,j \in \mathcal{Y}}$. 
    
\end{remark}


Note that morphisms equivalent to zero form a two-sided ideal under composition. Thus, we can define the quotient category $\kappa(\mathcal{Y},\Bbbk)$, to be the category with the same objects as $s(\mathcal{Y},\Bbbk)$ and the group of morphisms is given by $ \textup{Hom}_{\kappa(\mathcal{Y},\Bbbk)}(B,C) = {\textup{Hom}_{s(\mathcal{Y},\Bbbk)}(B ,C)}/\equiv$, for each $B,C\in\kappa(\mathcal{Y},\Bbbk)$. This category is an additive $\Bbbk$-category, and it is not abelian, to see this, it is enough to adapt the Remark~\ref{rem:nonabel}.

Additionally, note that the definitions in this paper on $s(\mathcal{Y},\Bbbk)$ are compatible with the equivalence relation $\equiv$, so we have a well-defined induced autofunctor $[-]:\kappa(\mathcal{Y},\Bbbk)\longrightarrow \kappa(\mathcal{Y},\Bbbk)$. The category $\kappa(\mathcal{Y},\Bbbk)$ is the previously  mentioned quotient of $s(\mathcal{Y},\Bbbk)$ which we want to give the triangulated structure. See \cite[Section 12.3, pp. 303--309]{DW17} or \cite[Section 1.1, pp. 1--9]{Hap88} for the definition and properties of Triangulated category. 

The next step for getting a triangulated structure on $\kappa(\mathcal{Y},\Bbbk)$ is to find a suitable set {\huge{$\tau$}} of \emph{distinguished triangle}, which we define to be triangles of the form 
	$
    \xymatrix{
    X\ar[r]^-{u} &
    Y\ar[r]^-{v} &
    Z\ar[r]^-{w} &
    [X]}
    $
in $\kappa(\mathcal{Y},\Bbbk)$ which are isomorphic to a standard triangle in $\kappa(\mathcal{Y},\Bbbk)$. In others words, there exists an isomophism of triangles in $\kappa(\mathcal{Y},\Bbbk)$
    $$
    \xymatrix{
    X\ar[r]^-{u}\ar[d]_-{\cong} &
    Y\ar[r]^-{v}\ar[d]_-{\cong} &
    Z\ar[r]^-{w}\ar[d]_-{\cong} &
    [X]\ar[d]_-{\cong}\\
    B\ar[r]_-{T} &
    C\ar[r]_-{\iota_C} &
    \mathbf{C}_T\ar[r]_-{\pi_B} &
    [B]}
    $$
for some morphism $T:B\longrightarrow C$ in $s(\mathcal{Y},\Bbbk)$. Denote the family of distinguished triangles by {\huge{$\tau$}}.

For a better exposition, we will denote
the triangle $
    \xymatrix{
    X\ar[r]^-{u} &
    Y\ar[r]^-{v} &
    Z\ar[r]^-{w} &
    [X]}
    $ 
by the sextuple $(X,Y,Z,u,v,w)$.

To show the rotation property for distinguished triangles, we need the following technical lemma.


\begin{lemma}\label{tr2:lemma}
    Let $T\in\textup{Hom}_{s(\mathcal{Y},\Bbbk)}(B,C)$ be a morphism. Then, 
    $$
    R=\left(
\begin{array}{ccc}
    -T^j_i & (\textup{Id}_B)^j_i& (\mathbf{0}^{C}_B)^j_i \\
\end{array}
 \right)_{i,j\in \mathcal{Y}}\in\textup{Hom}_{s(\mathcal{Y},\Bbbk)}([B],\mathbf{C}_{\iota_C})
 	$$
 	and
 	$$
 	S=\left(
\begin{array}{cc}
    (\mathbf{0}_C^B)^j_i\\ 
    (\textup{Id}_B)^j_i\\
    (\mathbf{0}_C^B)^j_i
\end{array}
 \right)_{i,j\in \mathcal{Y}}\in\textup{Hom}_{s(\mathcal{Y},\Bbbk)}(\mathbf{C}_{\iota_C},[B]).$$
\end{lemma}


\begin{proof}
We will only show that $R\in\textup{Hom}_{s(\mathcal{Y},\Bbbk)}([B],\mathbf{C}_{\iota_C})$, because for $S$ the proof is analogous. 
 
The properties (a), (c) and (d) of definition of morphism in $s(\mathcal{Y},\Bbbk)$,  follows from the fact that $T$ and $\textup{Id}_B$ are morphisms in $s(\mathcal{Y},\Bbbk)$. The equality $[B]R=R\mathbf{C}_{\iota_C}$ is consequence of the following computations 
{\footnotesize{\begin{align*}
    (-BR)_i^j&=\sum_{k\in \mathcal{Y}} -B^k_i\left(
\begin{array}{ccc}
    -T^j_k & (\textup{Id}_B)^j_k& (\mathbf{0}_B^C)^j_k \\
\end{array}
 \right)  \\
 & = \left(
\begin{array}{ccc}
    (BT)^j_i & -B^j_i & (\mathbf{0}_B^C)^j_i \\
\end{array}
 \right) = \left(
\begin{array}{ccc}
    (TC)^j_i& -B^j_i & (\mathbf{0}_B^C)^j_i \\
\end{array}
 \right)\\
 &=\sum_{k\in \mathcal{Y}} \left(
\begin{array}{ccc}
    -T^k_i & (\textup{Id}_B)^k_i& (\mathbf{0}^{C}_B)^k_i \\
\end{array}
 \right)\left(\begin{array}{cccc}
  -C^j_k   & (\mathbf{0}_{C}^B)^j_k &  (\textup{Id}_C)^j_k \\ 
(\mathbf{0}_{B}^C)^j_k      & -B^j_k & T^j_k \\
    (\mathbf{0}_{C}^C)^j_k  & (\mathbf{0}_{C}^B)^j_k & C^j_k    
\end{array}\right)\\ 
    & = (R\mathbf{C}_{\iota_C})_i^j,
\end{align*}}}
for all $i,j \in \mathcal{Y}$.

\end{proof}


\begin{proposition}
\label{prop:TR2}
If $(X,Y,Z,u,v,w)$ is a distinguished triangle, then the triangle $(Y,Z,[X],v,w,-[u])$ is distinguished.
\end{proposition}


\begin{proof}

Since the rotation property is compatible with isomorphisms of triangles, it is enough to prove  for a standard triangle $
    \xymatrix{
    B\ar[r]^-{T} &
    C\ar[r]^-{\iota_C} &
    \mathbf{C}_T\ar[r]^-{\pi_B} &
    [B]}
    $. In other words, we shall prove that
    $
    \xymatrix{
    C\ar[r]^-{\iota_C} &
    \mathbf{C}_T\ar[r]^-{\pi_B} &
    [B]\ar[r]^-{-[T]} &
    [C]}
    $ is a distinguished triangle. 
    
Consider the following diagram
    $$
    {\footnotesize{
    \xymatrix{
    C\ar[r]^{\iota_C}\ar[d]_{\textup{Id}_C} & \mathbf{C}_T\ar[r]^-{\iota_{\mathbf{C}_T}}\ar[d]_{\textup{Id}_{\mathbf{C}_T}} & \mathbf{C}_{\iota_C}\ar[r]^{\pi_C}\ar[d]_{S}   & [C]\ar[d]^{\textup{Id}_C}\\
    C\ar[r]_{\iota_C} & \mathbf{C}_T\ar[r]_-{\pi_B} & [B]\ar[r]_{-[T]}   & [C]}}}
    $$
where $S$ is the morphism given in Lemma \ref{tr2:lemma}. The commutativity in $\kappa(\mathcal{Y},\Bbbk)$ of the latter diagram follows to the fact that $\iota_{\mathbf{C}_T}S=\pi_B$ and $\pi_C+ST=\mathbf{C}_{\iota_C}K-KC$, where 
    $$
    \footnotesize{K=\left(
\begin{array}{cc}
     (\mathbf{0}_{C}^{C})^j_i \\
       (\mathbf{0}_{B}^{C})^j_i \\
       (\textup{Id}_C)^j_i
\end{array}
\right)_{i,j\in \mathcal{Y}}}.
    $$

To show that $S$ is an isomorphism in $\kappa(\mathcal{Y},\Bbbk)$, it is enough to consider the morphism $R$ introduced in Lemma~\ref{tr2:lemma}, and then note that $RS=\textup{Id}_{B}$ and
$\textup{Id}_{\mathbf{C}_{\iota_C}} -SR=K\mathbf{C}_{\iota_C}+\mathbf{C}_{\iota_C}K$, where 
$$
\footnotesize{K=\left(\begin{array}{cccc}
 (\mathbf{0}_{C}^C)^j_i   & (\mathbf{0}_{C}^B)^j_i &   (\mathbf{0}_C^C)^j_i\\ 
 (\mathbf{0}_{B}^C)^j_i     & (\mathbf{0}_{B}^B)^j_i & (\mathbf{0}^{C}_B)^j_i \\
    (\textup{Id}_C)^j_i  & (\mathbf{0}_{C}^B)^j_i & (\mathbf{0}_{C}^C)^j_i  
\end{array}\right)_{i,j\in \mathcal{Y}}}.
$$
\end{proof}


Remark~\ref{rem:TR1-2}, Proposition~\ref{prop:TR2} and the following proposition allow us to guarantee that $\kappa(\mathcal{Y},\Bbbk)$ is a pretriangulated category.

\begin{proposition}
\label{prop:TR3}
If $(X,Y,Z,u,v,w)$ and  $(X',Y',Z',u',v',w')$ are distinguished triangles, then for any pair $f\in \textup{Hom}_{\kappa(\mathcal{Y},\Bbbk)}(X,X')$ and  $g\in\textup{Hom}_{\kappa(\mathcal{Y},\Bbbk)}(Y,Y')$ of morphisms such that $fu'=ug$, there exists a morphism $h\in\textup{Hom}_{\kappa(\mathcal{Y},\Bbbk)}(Z,Z')$ such that the following diagram commutes in  $\kappa(\mathcal{Y},\Bbbk)$
    $$
    \xymatrix{X \ar[d]_{ f}\ar[r]^{u} &{Y}\ar[d]_{g}\ar[r]^{v} & Z\ar@{-->}[d]_{\exists h}\ar[r]^{ w} &[X]\ar[d]^{[f]} \\
    X'\ar[r]_{u'} & Y'\ar[r]_{v'} & Z'\ar[r]_{w'} & [X']}
    $$
\end{proposition}


\begin{proof}

Again, it suffices to prove this proposition for standard triangles. By assumption we have a diagram
    $$
    \xymatrix{A \ar[d]_{F}\ar[r]^{T} & {B}\ar[d]_{G}\ar[r]^{ \iota_{B}} & {\mathbf{C}_T}\ar[r]^{ \pi_A} & {[A]}\ar[d]^{F} \\
    {A'}\ar[r]_{T'} & {B'}\ar[r]_{ \iota_{B'}} & {\mathbf{C}_{T'}}\ar[r]_{ \pi_{A'}} & {[A']}}
    $$
where the left square commutes in $\kappa(\mathcal{Y},\Bbbk)$. This implies that there exists a $\kappa$-matrix $K$ such that $FT'\equiv TG$. From property (i) of $\kappa$-matrix we can consider the matrix
    $$
    H = \left(\begin{array}{cc}
   F^{j}_{i}  & K^{j}_{i}  \\
    (\mathbf{0}_{B}^{A'})^j_i &  G^{j}_{i} 
\end{array}\right)_{{i},{j} \in \mathcal{Y}}.
    $$
Let us to show that $H\in\textup{Hom}_{s(\mathcal{Y},\Bbbk)}(\mathbf{C}_T,\mathbf{C}_{T'} )$. The properties (a), (c) and (d) of the definition of morphism in $s(\mathcal{Y},\Bbbk)$,  follows from the fact that $F$ and $G$ are morphisms in $s(\mathcal{Y},\Bbbk)$ and $K$ is a $\kappa$-matrix. The equality  $\mathbf{C}_TH=H\mathbf{C}_{T'}$ is consequence of $FT'-TG=AK+KB'$, $AF=FA'$ and $BG=GB'$, because
{\footnotesize{\begin{align*}
\left(\mathbf{C}_TH\right)_i^j&=
\sum_{k\in \mathcal{Y }}\left(\begin{array}{cc}
    -A^{k}_{i} & T^{k}_{i}  \\
    (\mathbf{0}^{A}_{B})^k_i &  B^{k}_{i}
\end{array}\right)\left(\begin{array}{cc}
   F^{j}_{k}  & K^{j}_{k}  \\
   (\mathbf{0}_{B}^{A'})^j_k &   G^{j}_{k}
\end{array}\right)
= \left(\begin{array}{cc}
    -(AF)^{j}_{i} & (TG)^{j}_{i}-(AK)^{j}_{i}  \\
    (\mathbf{0}^{A'}_{B})^j_i &  (BG)^{j}_{i}
\end{array}\right)
\end{align*}}}
and
{\footnotesize{\begin{align*}
(H \mathbf{C}_{T'})^j_i &= 
\sum_{k\in \mathcal{Y}}\left(\begin{array}{cc}
  F^{k}_{i}  & K^{k}_{i}  \\
    (\mathbf{0}_{B}^{A'})^k_i &  G^{k}_{i}
\end{array}\right)\left(\begin{array}{cc}
    -A'^{j}_{k} & T'^{j}_{k}  \\
    (\mathbf{0}^{A'}_{B'})^j_k &  B'^{j}_{k}
\end{array}\right)= \left(\begin{array}{cc}
    -(FA')^{j}_{i} & (FT')^{j}_{i}+(KB')^{j}_{i}  \\
  (\mathbf{0}^{A'}_{B})^j_i&  (GB')^{j}_{i}
\end{array}\right),
\end{align*}}}
 for all $i,j \in \mathcal{Y}$.

The commutativity of the following  diagram 
    $$\xymatrix{{A} \ar[d]_{ F}\ar[r]^{ T} & {B}\ar[d]_{G}\ar[r]^{ \iota_{B}} & {\mathbf{C}_T}\ar@{-->}[d]_{H}\ar[r]^{ \pi_A} &{[A]}\ar[d]^{F} \\
{A'}\ar[r]_{T'} & {B'}\ar[r]_{\iota_{B'}} & {\mathbf{C}_{T'}}\ar[r]_{\pi_{A'}} & {[A']}}$$
 follows of the fact that 
 {\footnotesize{\begin{align*}
        (G\iota_{B'})_i^j& = \sum_{k\in \mathcal{Y}}G^k_{i}\left(\begin{matrix}
          (\mathbf{0}^{A'}_{B'})^j_k & (\textup{Id}_{B'})^{j}_k
\end{matrix}\right)
= \left(\begin{matrix}
            (\mathbf{0}^{A'}_{B})^j_i  & G^{j}_{i}
\end{matrix}\right)=\sum_{k \in \mathcal{Y}}\left(\begin{matrix}
            (\mathbf{0}^{A}_{B})^k_i  & (\textup{Id}_{B})^{k}_{i}
\end{matrix}\right)\left(\begin{array}{cc}
   F^{j}_{k}  & K^{j}_{k}  \\
    (\mathbf{0}^{A'}_{B})^j_k &  G^{j}_{k}
\end{array}\right)\\
&=(\iota_{B}H)_i^j,\\
   (\pi_A F)_i^j & = \left(\sum_{k\in \mathcal{Y}}\begin{matrix}
             (\textup{Id}_{A})^k_{i}\\
              (\mathbf{0}^{A}_{B})^k_i
\end{matrix}\right)F^{j}_k = \left(\begin{matrix}
       F^{j}_i \\
              (\mathbf{0}^{A'}_{B})^j_i
\end{matrix}\right) 
    = \sum_{k \in \mathcal{Y}}\left(\begin{array}{cc}
   F^{k}_{i}  & K^{k}_{i}  \\
     (\mathbf{0}_{B}^{A'})^k_i &  G^{k}_{i}
\end{array}\right)\left(\begin{matrix}
            (\textup{Id}_{A'})^{j}_{k}\\
             (\mathbf{0}^{A'}_{B'})_k^j		
\end{matrix}\right)
    = (H\pi_{A'})_i^j.
\end{align*}}}
for all $i,j\in\mathcal{Y}$.

\end{proof}


In order to show that $\kappa(\mathcal{Y},\Bbbk)$ is a triangulated category, it remains to show the octahedral axiom.


\begin{proposition}
\label{prop:Octae}
   For any $(X, Y, X', u, u', v')$, $(Y, Z, Z', v,w, w')$ and $(X, Z, Y', uv, p,q)$ distinguished triangles, there exists a distinguished triangle $(X',Y',Z',f,g,w'u')$  making the following diagram commutative in  $\kappa(\mathcal{Y},\Bbbk)$
    
    $$
    {\footnotesize{\xymatrix{
    X\ar[rr]^-{u}\ar[d]_-{\textup{Id}_X}   &&Y\ar[rr]^-{u'}\ar[d]_-{v}    &&X'\ar@{-->}[d]_-{f}\ar[rr]^<<<<<{v'} &    &[X]\ar[d]^-{\textup{Id}_X}\\
X\ar[rr]^-{uv}\ar[d]_-{u}  & &Z\ar[rr]^-{p}\ar[d]_{\textup{Id}_Z}    &&Y'{ \ar@{-->}[d]_-{g}}\ar[rr]^-{q}    & &[X]\ar[d]^-{u} \\   
Y\ar[rr]^-{v}  & &Z\ar[rr]^-{w}    &&Z'\ar[d]_-{w'u'}\ar[rr]^-{w'}    & &[Y]\ar[lld]^-{u'} \\ 
&& && [X']&&}}}
    $$
\end{proposition}


\begin{proof}

Again, it suffices to prove the Octahedral axiom for standard triangles. First of all, let us show that there exists morphisms $F\in\textup{Hom}_{s(\mathcal{Y},\Bbbk)}(\mathbf{C}_S, \mathbf{C}_{ST})$ and $G\in\textup{Hom}_{s(\mathcal{Y},\Bbbk)}(\mathbf{C}_{ST}, \mathbf{C}_{T})$  such that the following diagram commutes in $\kappa(\mathcal{Y},\Bbbk)$

 $${\footnotesize{\xymatrix{A\ar[rr]^{S}\ar[d]_{\textup{Id}_A}   &&B\ar[rr]^{\iota_B}\ar[d]_{T}    &&\mathbf{C}_{S}\ar@{-->}[d]_{\exists F}\ar[rr]^-{\pi_A} &    &[A]\ar[d]^{\textup{Id}_A}\\
A\ar[rr]^{ST}\ar[d]_{S}  & &C\ar[rr]^-{\overline{\iota_C}}\ar[d]_-{\textup{Id}_C}    &&\mathbf{C}_{ST}\ar@{-->}[d]_{\exists G}\ar[rr]^-{\overline{\pi_{A}}}    & & [A]\ar[d]^{S} \\   
B\ar[rr]_{T}  & &C\ar[rr]_-{{\iota_C}}    &&\mathbf{C}_T\ar[rr]_-{\pi_B}    & &[B]
}}}$$
where, to avoid confusion,  $\overline{\iota_C}$ and  $\overline{\pi_A}$ denote the morphism $\iota_C:C\longrightarrow \mathbf{C}_{ST}$ and  $\pi_A:\mathbf{C}_{ST}\longrightarrow [A]$ defined to Lemma~\ref{lem:morph}.

Consider the matrices  
    $$
F = \left(\begin{array}{cc}
   (\textup{Id}_{A})^j_i & (\mathbf{0}_A^C)^j_i  \\
    (\mathbf{0}_B^A)^j_i & T^j_i  
\end{array}\right)_{i,j \in \mathcal{Y}}
\ \ \ \ \ \text{and }\ \ \ \ \ G = \left(\begin{array}{cc}
   S^j_i  & (\mathbf{0}_A^C)^j_i  \\
    (\mathbf{0}^B_C)^j_i & (\textup{Id}_C)^j_i  
\end{array}\right)_{i,j \in \mathcal{Y}}.
	$$
We only show that $F\in\textup{Hom}_{s(\mathcal{Y},\Bbbk)}(\mathbf{C}_S, \mathbf{C}_{ST})$, the proof for $G\in\textup{Hom}_{s(\mathcal{Y},\Bbbk)}(\mathbf{C}_{ST}, \mathbf{C}_{T})$ is analogous. The properties (a), (c) and (d) of the definition of morphism in $s(\mathcal{Y},\Bbbk)$,  follows from the fact that $T$ and $\textup{Id}_A$ are morphisms in $s(\mathcal{Y},\Bbbk)$. The equality $\mathbf{C}_SF=F\mathbf{C}_{ST}$ is a consequence of the following computations 
{\footnotesize{\begin{align*}
   ( \mathbf{C}_{S}F)^j_i&= \sum_{k\in \mathcal{Y}}\left(\begin{array}{cc}
   -A^k_i  & S^k_i  \\
    (\mathbf{0}^A_B)^k_i & B^k_i  
\end{array}\right)\left(\begin{array}{cc}
   (\textup{Id}_{A})^j_k & (\mathbf{0}_A^C)^j_k  \\
    (\mathbf{0}_B^A)^j_k & T^j_k  
\end{array}\right)
=\left(\begin{array}{cc}
   -A^j_i & (ST)^j_i  \\
    (\mathbf{0}^A_B)^j_i& (BT)^j_i  
\end{array}\right),
\end{align*}}}
{\footnotesize{\begin{align*}
   ( F\mathbf{C}_{ST})^j_i&= \sum_{k \in \mathcal{Y}}\left(\begin{array}{cc}
   (\textup{Id}_A)^k_i  & (\mathbf{0}^C_A)^k_i  \\
    (\mathbf{0}^A_B)^k_i & T^k_i 
\end{array}\right)\left(\begin{array}{cc}
   -A^j_k   & (ST)^j_k   \\
    (\mathbf{0}^A_C)^j_k & C^j_k   
\end{array}\right)=\left(\begin{array}{cc}
   -A^j_i  & (ST)^j_i  \\
     (\mathbf{0}^A_B)^j_i & (TC)^j_i  
\end{array}\right),
\end{align*}}}
for all $i,j \in \mathcal{Y}$.

The commutativity is straightforward by multiplication of matrices, for instance,
{\footnotesize{\begin{align*}
    (\iota_BF)^j_i &=\sum_{k\in \mathcal{Y}}\left(\begin{array}{cc}
   (\mathbf{0}_B^A)^k_i & (\text{Id}_{B})^k_i  \\
\end{array}\right)\left(\begin{array}{cc}
   (\textup{Id}_A)^j_k  & (\mathbf{0}^C_A)^j_k  \\
    (\mathbf{0}^A_B)^j_k & T^j_k 
\end{array}\right)=\left(\begin{array}{cc}
   (\mathbf{0}_B^A)^j_i & T^j_i
\end{array}\right)\\
&=\sum_{k\in \mathcal{Y}}T_i^k\left(\begin{array}{cc}
   (\mathbf{0}_C^A)^j_k & (\textup{Id}_C)^j_k
\end{array}\right)= (T\overline{\iota_{C}})^j_i,
\end{align*}}}
and  
{\footnotesize{\begin{align*}
    (G\pi_B)^j_i &= \sum_{k\in \mathcal{Y}}\left(\begin{array}{cc}
       S^k_i  & (\mathbf{0}^B_A)^k_i  \\
    (\mathbf{0}_C^B)^k_i & (\text{Id}_C)^k_i  
\end{array}\right)\left(\begin{array}{cc}
   (\text{Id}_{B})^j_k \\
    (\mathbf{0}_C^B)^j_k 
\end{array}\right)=\left(\begin{array}{cc}
    S^j_i \\
    (\mathbf{0}_C^B)^j_i
\end{array}\right)=\sum_{k\in \mathcal{Y}}\left(\begin{array}{cc}
    (\textup{Id}_A)^k_i \\
    (\mathbf{0}_C^A)^k_i
\end{array}\right)S^j_k=(\overline{\pi_A}S)^j_i.
\end{align*}}}
 for all $i,j \in \mathcal{Y}$.
 
 Finally, let us show that
    $$
    \Lambda=\left(\begin{array}{ccccc}
     (\mathbf{0}_{B}^A)^j_i& (\textup{Id}_B)^j_i &
     (\mathbf{0}_{B}^A)^j_i& (\mathbf{0}^C_B)^j_i\\
     (\mathbf{0}^A_C)^j_i& (\mathbf{0}^B_C)^j_i &
     (\mathbf{0}^A_C)^j_i & (\textup{Id}_C)^j_i 
\end{array}\right)_{i,j \in \mathcal{Y}}\in\textup{Hom}_{s(\mathcal{Y},\Bbbk)}(\mathbf{C}_T,\mathbf{C}_F)
    $$
and that it is an isomorphism in $\kappa(\mathcal{Y},\Bbbk)$ such that the following diagram commutes in $\kappa(\mathcal{Y},\Bbbk)$
{\footnotesize{$$\xymatrix@C=15pt{	
	{\mathbf{C}_S}\ar[rr]^-{F}\ar[d]_-{\textup{Id}_{\mathbf{C}_S}}	&&	{\mathbf{C}_{ST}}\ar[rr]^-{G}\ar[d]_-{\textup{Id}_{\mathbf{C}_{ST}}}		&&	
 {\mathbf{C}_{T}}\ar[rr]^-{\pi_B\iota_B}\ar@{-->}[d]_-{\exists \Lambda}	&	&{[\mathbf{C}_S]}\ar[d]^-{\textup{Id}_{\mathbf{C}_S}}		\\
	 {\mathbf{C}_S}\ar[rr]_-{F}&&{\mathbf{C}_{ST}}\ar[rr]_-{\iota_{\mathbf{C}_{ST}}} &&
 {\mathbf{C}_{F}}\ar[rr]_-{\pi_{\mathbf{C}_S}}	&&{[\mathbf{C}_S]}}$$}}

The properties (a), (c) and (d) of the definition of morphism in $s(\mathcal{Y},\Bbbk)$,  follows from the fact that $\textup{Id}_{\mathbf{C}_S}$ and $\textup{Id}_{\mathbf{C}_{ST}}$ are morphisms in $s(\mathcal{Y},\Bbbk)$. The equality $\mathbf{C}_T\Lambda=\Lambda\mathbf{C}_{F}$ is a consequence of the following computations 
{\footnotesize{\begin{align*}
    (\Lambda\mathbf{C}_{F} )^j_i& =\sum_{k\in \mathcal{Y}}\left(\begin{array}{ccccc}
     (\mathbf{0}_{B}^A)^k_i& (\textup{Id}_B)^k_i &
     (\mathbf{0}_{B}^A)^k_i& (\mathbf{0}^C_B)^k_i\\
     (\mathbf{0}^A_C)^k_i& (\mathbf{0}^B_C)^k_i &
     (\mathbf{0}^A_C)^k_i & (\textup{Id}_C)^k_i 
\end{array}\right)\left(\begin{array}{cc}
      -(\mathbf{C}_S)^j_k    & F^j_k \\
        (\mathbf{0}^{\mathbf{C}_S}_{\mathbf{C}_{ST}})^j_k  & (\mathbf{C}_{ST})^j_k
     \end{array}\right)\\
    &=\sum_{k\in \mathcal{Y}}\left(\begin{array}{ccccc}
     (\mathbf{0}_{B}^A)^k_i& (\textup{Id}_B)^k_i &
     (\mathbf{0}_{B}^A)^k_i& (\mathbf{0}^C_B)^k_i\\
     (\mathbf{0}^A_C)^k_i& (\mathbf{0}^B_C)^k_i &
     (\mathbf{0}^A_C)^k_i & (\textup{Id}_C)^k_i 
\end{array}\right)\left(\begin{array}{cccc}
        A^j_k  & -S^j_k& (\textup{Id}_A)^j_k& (\mathbf{0}_{A}^C)^j_k  \\
          (\mathbf{0}^{A}_B)^j_k  &-B^j_k & (\mathbf{0}^{A}_B)^j_k & T^j_k \\
             (\mathbf{0}^{A}_A)^j_k &(\mathbf{0}^{B}_A)^j_k &-A^j_k & (ST)^j_k  \\
                 (\mathbf{0}^{A}_C)^j_k& (\mathbf{0}^{B}_C)^j_k&(\mathbf{0}^{A}_C)^j_k &C^j_k  \\
     \end{array}\right)\\
     &=\left(\begin{array}{ccccc}
     (\mathbf{0}_{B}^A)^j_i& -B^j_i &
     (\mathbf{0}_{B}^A)^j_i& T^j_i\\
     (\mathbf{0}^A_C)^j_i& (\mathbf{0}^B_C)^j_i &
     (\mathbf{0}^A_C)^j_i & C^j_i 
\end{array}\right)\\
&=\sum_{k\in \mathcal{Y}}\left(\begin{array}{cc}
    -B^k_i & T^k_i  \\
    (\mathbf{0}^B_C)^k_i &  C^k_i
\end{array}\right)\left(\begin{array}{ccccc}
     (\mathbf{0}_{B}^A)^j_k& (\textup{Id}_B)^j_k &
     (\mathbf{0}_{B}^A)^j_k& (\mathbf{0}^C_B)^j_k\\
     (\mathbf{0}^A_C)^j_k& (\mathbf{0}^B_C)^j_k &
     (\mathbf{0}^A_C)^j_k & (\textup{Id}_C)^j_k 
\end{array}\right)\\
&=( \mathbf{C}_{T}\Lambda)^j_i,
\end{align*}}}
for all $i,j \in \mathcal{Y}$.

The commutativity is straightforward by multiplication of matrices, for instance 
{\footnotesize{\begin{align*}
(\iota_{\mathbf{C}_{ST}}-G\Lambda)^j_i & =\left(\begin{array}{ccccc}   (\mathbf{0}_{\mathbf{C}_{ST}}^{\mathbf{C}_{S}})^j_i& (\textup{Id}_{\mathbf{C}_{ST}})^j_i
\end{array}\right) \\
&\ \ \ \ \ -\sum_{k\in \mathcal{Y}}\left(\begin{array}{cc}
    S^k_i & (\mathbf{0}^C_A)^k_i  \\
    (\mathbf{0}^B_C)^k_i &  (\textup{Id}_C)^k_i
\end{array}\right)\left(\begin{array}{ccccc}
     (\mathbf{0}_{B}^A)^j_k& (\textup{Id}_B)^j_k &
     (\mathbf{0}_{B}^A)^j_k& (\mathbf{0}^C_A)^j_k\\
     (\mathbf{0}^A_C)^j_k& (\mathbf{0}^B_C)^j_k &
     (\mathbf{0}^A_C)^j_k & (\textup{Id}_C)^j_k 
\end{array}\right)\\
&=\left(\begin{array}{ccccc}
     (\mathbf{0}_{A}^A)^j_i& (\mathbf{0}_{A}^B )^j_i &
     (\textup{Id}_A)^j_i& (\mathbf{0}^C_A)^j_i\\
     (\mathbf{0}^A_C)^j_i& (\mathbf{0}^B_C)^j_i &
     (\mathbf{0}^A_C)^j_i & (\textup{Id}_C)^j_i 
\end{array}\right)-\left(\begin{array}{ccccc}
     (\mathbf{0}_{A}^A)^j_i& S^j_i &
     (\mathbf{0}_{A}^A)^j_i& (\mathbf{0}_{A}^C)^j_i\\
     (\mathbf{0}^A_C)^j_i& (\mathbf{0}^B_C)^j_i &
     (\mathbf{0}^A_C)^j_i & (\textup{Id}_C)^j_i 
\end{array}\right)\\
&= \left(\begin{array}{ccccc}
     (\mathbf{0}_{A}^A)^j_i& -S^j_i &
     (\textup{Id}_A)^j_i& (\mathbf{0}_{A}^C)^j_i\\
     (\mathbf{0}^A_C)^j_i& (\mathbf{0}^B_C)^j_i &
     (\mathbf{0}^A_C)^j_i & (\mathbf{0}^C_C)^j_i 
\end{array}\right)\\
&=\left(\begin{array}{ccccc}
     -A^j_i & (\mathbf{0}_A^B)^j_i &
    (\mathbf{0}_{A}^A)^j_i & (\mathbf{0}^C_A)^j_i\\
     (\mathbf{0}^A_C)^j_i& (\mathbf{0}^B_C)^j_i &
     (\mathbf{0}^A_C)^j_i & (\mathbf{0}_{C}^C)^j_i 
\end{array}\right)+\left(\begin{array}{ccccc}
     A^j_i & -S^j_i &
    (\textup{Id}_A)^j_i & (\mathbf{0}^C_A)^j_i\\
     (\mathbf{0}^A_C)^j_i& (\mathbf{0}^B_C)^j_i &
     (\mathbf{0}^A_C)^j_i & (\mathbf{0}_{C}^C)^j_i 
\end{array}\right)\\
&=\sum_{k \in \mathcal{Y}}\left(\begin{array}{cc}
   -A^k_i   & (ST)^k_i   \\
    (\mathbf{0}^A_C)^k_i & C^k_i   
\end{array}\right)\left(\begin{array}{ccccc}
    (\textup{Id}_A)^j_k & (\mathbf{0}_{A}^B )^j_k &
     (\mathbf{0}_{A}^A)^j_k& (\mathbf{0}^C_A)^j_k\\
     (\mathbf{0}^A_C)^j_k& (\mathbf{0}^B_C)^j_k &
     (\mathbf{0}^A_C)^j_k & (\mathbf{0}_{C}^C)^j_k
\end{array}\right)\\
&+\sum_{k \in \mathcal{Y}}\left(\begin{array}{ccccc}
     (\textup{Id}_A)^k_i & (\mathbf{0}_{A}^B )^k_i &
     (\mathbf{0}_{A}^A)^k_i& (\mathbf{0}^C_A)^k_i\\
     (\mathbf{0}^A_C)^k_i& (\mathbf{0}^B_C)^k_i &
     (\mathbf{0}^A_C)^k_i & (\mathbf{0}_{C}^C)^k_i 
\end{array}\right)\left(\begin{array}{cccc}
        A^j_k  & -S^j_k& (\textup{Id}_A)^j_k& (\mathbf{0}_{A}^C)^j_k  \\
          (\mathbf{0}^{A}_B)^j_k  &-B^j_k &(\mathbf{0}^{A}_B)^j_k & T^j_k \\
             (\mathbf{0}^{A}_A)^j_k &(\mathbf{0}^{B}_A)^j_k &-A^j_k & (ST)^j_k  \\
                 (\mathbf{0}^{A}_C)^j_k& (\mathbf{0}^{B}_C)^j_k&(\mathbf{0}^{A}_C)^j_k &C^j_k  \\
     \end{array}\right)\\
     &=(\mathbf{C}_{ST}K+K\mathbf{C}_{F})^j_i,
\end{align*}}}
for all $i,j \in \mathcal{Y}$, where 
    \begin{align*}
    K=\left(\begin{array}{ccccc}
     (\textup{Id}_A)^j_i& (\mathbf{0}_{A}^B )^j_i &
     (\mathbf{0}_A^A)^j_i & (\mathbf{0}^C_A)^j_i\\
     (\mathbf{0}^A_C)^j_i& (\mathbf{0}^B_C)^j_i &
     (\mathbf{0}^A_C)^j_i & (\mathbf{0}_C^C)^j_i 
\end{array}\right)_{i,j \in \mathcal{Y}}
\end{align*}
is a $\kappa$-matrix. Moreover
{\footnotesize{\begin{align*}
    (\Lambda\pi_{\mathbf{C}_{S}})^j_i&= \sum_{k\in \mathcal{Y}}\left(\begin{array}{ccccc}
     (\mathbf{0}_{B}^A)^k_i& (\textup{Id}_B)^k_i &
     (\mathbf{0}_{B}^A)^k_i& (\mathbf{0}^C_B)^k_i\\
     (\mathbf{0}^A_C)^k_i& (\mathbf{0}^B_C)^k_i &
     (\mathbf{0}^A_C)^k_i & (\textup{Id}_C)^k_i 
\end{array}\right)\left(\begin{array}{cc}
    (\textup{Id}_{\mathbf{C}_S})^j_k \\
    (\mathbf{0}^{\mathbf{C}_S}_{\mathbf{C}_{ST}})^j_k 
\end{array}\right)\\ 
&=\sum_{k\in \mathcal{Y}}\left(\begin{array}{ccccc}
     (\mathbf{0}_{B}^A)^k_i& (\textup{Id}_B)^k_i &
     (\mathbf{0}_{B}^A)^k_i& (\mathbf{0}^C_B)^k_i\\
     (\mathbf{0}^A_C)^k_i& (\mathbf{0}^B_C)^k_i &
     (\mathbf{0}^A_C)^k_i & (\textup{Id}_C)^k_i 
\end{array}\right)\left(\begin{array}{cc}
    (\textup{Id}_{A})^j_k & (\mathbf{0}_A^B)^j_k  \\
    (\mathbf{0}^A_B)^j_k & (\textup{Id}_{B})^j_k\\
      (\mathbf{0}_A^A)^j_k & (\mathbf{0}_A^B)^j_k  \\
    (\mathbf{0}^A_C)^j_k & (\mathbf{0}_C^B)^j_k\\
\end{array}\right)\\
&= \left(\begin{array}{cc}
    (\mathbf{0}_B^A)^j_i & (\textup{Id}_{B})^j_i  \\
    (\mathbf{0}^A_C)^j_i & (\mathbf{0}_C^B)^j_i\\
    \end{array}\right)= \sum_{k\in \mathcal{Y}}\left(\begin{array}{cc}
   (\textup{Id}_{B})^k_i \\ (\mathbf{0}^B_C)^k_i 
\end{array}\right)\left(\begin{array}{cc}
  (\mathbf{0}^B_A)^j_k & (\textup{Id}_{B})^j_k  
    \end{array}\right)\\
    & = (\pi_B\iota_B)^j_i,
\end{align*}}}
for all $i,j \in \mathcal{Y}$. Therefore, the proof is done, since $\textup{Id}_{\mathbf{C}_S}$ and $\textup{Id}_{\mathbf{C}_{ST}}$ are isomorphism in $\kappa(\mathcal{Y},\Bbbk)$ and five lemma for pretriangulated category.

\end{proof}


From Remark~\ref{rem:TR1-2} and propositions~\ref{prop:TR2}, \ref{prop:TR3} and \ref{prop:Octae} we can conclude the main result of this section.


\begin{theorem}\label{main 1}
    The category $\kappa(\mathcal{Y},\Bbbk)$ with the autofunctor $[-]:\kappa(\mathcal{Y},\Bbbk)\longrightarrow \kappa(\mathcal{Y},\Bbbk)$ and the family of distinguised triangles {\huge{$\tau$}} is a triangulated category.
\end{theorem}



\section*{Acknowledgements}
The first author was partially supported by the Coordena\c{c}\~{a}o de Aperfei\c coamento de Pessoal de N\'ivel Superior -- Brasil (CAPES) --  Finance Code 001. The second author has been partially supported by Fundação de Amparo à Pesquisa do Estado do Amazonas (FAPEAM). The third author has been partially supported by Conselho Nacional de Desenvolvimento Cient\'ifico e Tecnol\'ogico (CNPq) (Grant $\#$140582/2021-5).




\end{document}